\documentclass[11pt]{article}%
\usepackage{amssymb}
\usepackage{eurosym}
\usepackage{amsfonts}
\usepackage{amsmath}
\usepackage{graphicx}
\usepackage{epstopdf}%
\providecommand{\U}[1]{\protect\rule{.1in}{.1in}}
\newtheorem{theorem}{Theorem}
\newtheorem{acknowledgement}[theorem]{Acknowledgement}

\newtheorem{definition}[theorem]{Definition}

\newtheorem{proposition}[theorem]{Proposition}

\newenvironment{proof}[1][Proof]{\noindent\textbf{#1.} }{\ \rule{0.5em}{0.5em}}
\begin{document}

\title{Fingerprints of Closed Trajectories of a Strebel Quadratic Differential}
\author{Gliaa Braek and Faouzi Thabet\\University of Gabes, Tunisia}
\maketitle

\begin{abstract}
In this note, we study the fingerprints of closed smooth curves that are
trajectories of a particular Strebel quadratic differentials. It is a
generalization of the cases of polynomials and rational lemniscates.

\end{abstract}


\textit{Keywords and phrases: Stebel Quadratic differentials.Lemniscates.
Conformal Maps.Fingerprints.}

\section{\bigskip A Strebel quadratic differential}

\bigskip We consider the quadratic differential on the Riemann sphere
$\widehat{%
\mathbb{C}
}$ :
\[
\varpi\left(  z\right)  =-\left(  \sum_{i=1}^{n}\frac{\alpha_{i}}{z-a_{i}%
}\right)  ^{2}dz^{2},
\]
where the $a_{i}$ 's are pairwise distinct complex numbers, the $\alpha_{i}$
's are real numbers with non-negative sum $\alpha.$

\emph{Finite critical points} and \emph{infinite critical points }of the
quadratic differential $\varpi$ are respectively its zero's and poles
($\left\{  \infty\right\}  \cup\left\{  a_{j}\right\}  _{1\leq j\leq n}$); all
other points of $\widehat{%
\mathbb{C}
}$ are called \emph{regular points} of $\varpi.$

\emph{Horizontal trajectories} (or just trajectories) of the quadratic
differential $\varpi$ are the zero loci of the equation%
\[
\varpi\left(  z\right)  >0,
\]
or equivalently%
\begin{equation}
\prod_{i=1}^{n}\left\vert z-a_{i}\right\vert ^{\alpha_{i}}=\text{\emph{const}%
}. \label{eq traj}%
\end{equation}
For a given non-negative real $\lambda,$ we define the \emph{lemniscate
}$\Gamma_{\lambda}$ as the trajectory of the quadratic differential $\varpi$
with $\emph{const}=\lambda$ in (\ref{eq traj}) : \emph{ }
\[
\Gamma_{\lambda}=\left\{  z\in%
\mathbb{C}
;\left\vert f\left(  z\right)  \right\vert =\prod_{i=1}^{n}\left\vert
z-a_{i}\right\vert ^{\alpha_{i}}=\lambda\right\}  ,
\]
where $f\left(  z\right)  $ is the multi-valued function defined by
\[
f\left(  z\right)  =\prod_{i=1}^{n}\left(  z-a_{i}\right)  ^{\alpha_{i}}.
\]
Any connected component of $\Gamma_{\lambda}$ will be called a
\emph{sub-lemniscate.}

The \emph{vertical} (or, \emph{orthogonal}) trajectories are obtained by
replacing $\Im$ by $\Re$ in equation (\ref{eq traj}). The horizontal and
vertical trajectories of the quadratic differential $\varpi$ produce two
pairwise orthogonal foliations of the Riemann sphere $\widehat{%
\mathbb{C}
}$.

A trajectory passing through a critical point of $\varpi$ is called
\emph{critical trajectory}. In particular, if it starts and ends at a finite
critical point, it is called \emph{finite critical trajectory}, otherwise, we
call it an \emph{infinite critical trajectory}. \bigskip If two different
trajectories are not disjoint, then their intersection must be a zero of the
quadratic differential.

The closure of the set of finite and infinite critical trajectories is called
the \emph{critical graph} of $\varpi,$ we denote it by $\Gamma;$ see example
in Figure \ref{FIG1}. \begin{figure}[th]
\centering\includegraphics[height=1.8in,width=2.8in]{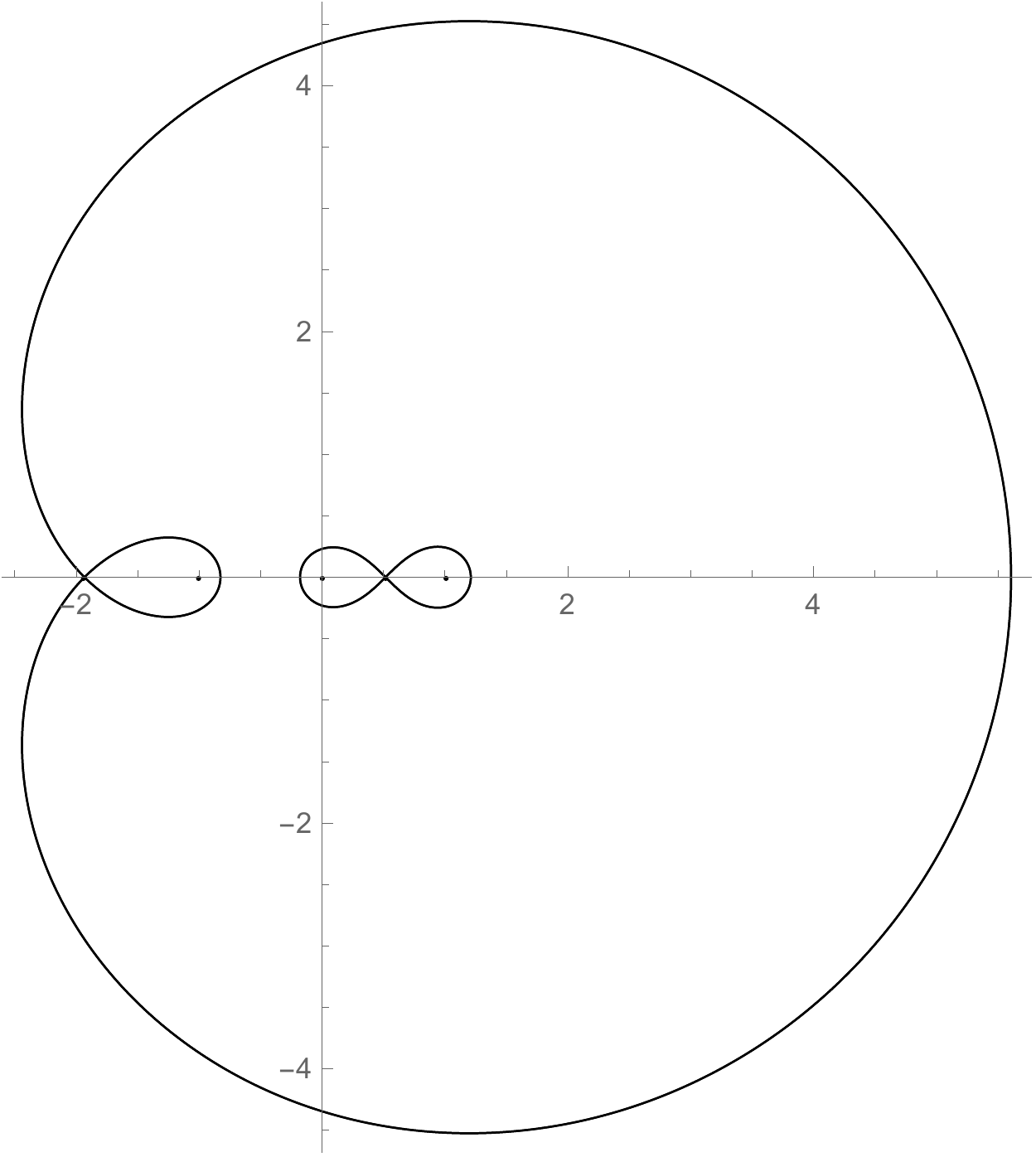}\caption{Critical
graph of $-\left(  \frac{1}{z-1}-\frac{1}{z+1}+\frac{\sqrt{2}}{z}\right)
^{2}dz^{2}.$}%
\label{FIG1}%
\end{figure}

The local and global structures of the trajectories is well known (more
details about the theory of quadratic differentials can be found in
\cite{Strebel},\cite{jenkins}, or \cite{F.Thabet}). The set $\widehat{%
\mathbb{C}
}\setminus\Gamma_{f}$ consists of a finite number of domains called the
\emph{Configurations Domain} of $\varpi.$ For a general quadratic differential
on a $\widehat{%
\mathbb{C}
}$, there are five kind of Configuration Domains, see \cite[Theorem3.5]%
{jenkins}. Since all the infinite critical points of $\varpi$ are poles of
order $2$ with negative residues, then there are three possible Domain Configurations:

\begin{itemize}
\item the \emph{Circle Domain} : It is swept by closed trajectories and
contains exactly one double pole that we will call the center of the domain.
Its boundary is a closed critical trajectory. For a suitably chosen real
constant $c$ and some real number $r>0,$ the function $z\longmapsto
r\exp\left(  cf\left(  t\right)  \right)  $ is a conformal map from the circle
domain $D$ onto the unit disk; it extends continuously to the boundary
$\partial D,$ and sends the double pole to the origin.

\item the \emph{Ring Domain}: It is swept by closed trajectories. Its boundary
consists of two connected components. For a suitably chosen real constant $c$
and some real numbers $0<r_{1}<r_{2},$ the function $z\longmapsto\exp\left(
cf\left(  t\right)  \right)  $ is a conformal map from the circle domain $D$
onto the annulus $\left\{  z:r_{1}<\left\vert z\right\vert <r_{2}\right\}  $
and it extends continuously to the boundary $\partial D.$

\item the \emph{Dense domain : }It is swept by recurrent critical trajectory
i.e., the interior of its closure is non-empty.
\end{itemize}

\begin{definition}
\bigskip A quadratic differential is \emph{Strebel }if the set of its closed
trajectories covers $\widehat{%
\mathbb{C}
}$ up to a zero Lebesgue measure set.
\end{definition}

\begin{proposition}
\bigskip The quadratic differential $\varpi$ is Strebel.
\end{proposition}

\begin{proof}
Since the infinite critical points of the quadratic differential $\varpi$ are
only double poles with negative residues, it suffices to show that it has no
recurrent trajectories. Indeed,If $\varpi$ has a recurrent trajectory, then,
its Domain Configuration contains a dense domain $D$ :
\[
\prod_{i=1}^{n}\left\vert z-a_{i}\right\vert ^{\alpha_{i}}=\text{\emph{const}
on }D.
\]
The function
\[
z\longmapsto\Re\left(  \sum_{i=1}^{n}\alpha_{i}\log\left(  z-a_{i}\right)
\right)
\]
is continuous and constant on the open subset $D$ of $%
\mathbb{C}
,$ then it is constant everywhere in $%
\mathbb{C}
$, which is clearly impossible by harmonicity.
\end{proof}

Let $p\left(  z\right)  $ the complex polynomial satisfying%
\[
\sum_{i=1}^{n}\frac{\alpha_{i}}{z-a_{i}}=\frac{p\left(  z\right)  }%
{\prod_{i=1}^{n}\left(  z-a_{i}\right)  }.
\]
By assumption, the leading coefficient $\alpha=\sum_{i=1}^{n}\alpha_{i}\neq0$
of $p\left(  z\right)  ,$ then $\deg p=n-1.$ Let $z_{1},...,z_{n-1}$ be the
zeros (counted with multiplicities) of $p\left(  z\right)  ,$
\[
w_{1}=\prod_{i=1}^{n}\left\vert z_{1}-a_{i}\right\vert ^{\alpha_{i}%
},...,w_{n-1}=\prod_{i=1}^{n}\left\vert z_{n-1}-a_{i}\right\vert ^{\alpha_{i}%
}.
\]

\begin{proposition}
If for some $1$ $\leq i<j\leq n-1,$ we have%
\[
\left\vert w_{i}\right\vert =\left\vert w_{j}\right\vert =\max\left\{
\left\vert w_{k}\right\vert ;k=1,...,n-1\right\}  ,
\]
then, there exists a finite critical trajectory connecting $z_{i}$ and
$z_{j}.$ In particular, the critical graph $\Gamma_{f}$ is connected, if and
only if $\left\vert w_{1}\right\vert =\cdot\cdot\cdot=\left\vert
w_{n-1}\right\vert .$
\end{proposition}

\begin{proof}
If no finite critical trajectory connecting $z_{i}$ and $z_{j},$ then a
lemniscate $\Gamma_{c},$ for some $c>\left\vert w_{i}\right\vert ,$ is not
connected : $\Gamma_{c}$ is a disjoint union of $s\geq2$ closed curves
(sub-lemniscates) $L_{1},...,L_{s},$ each of them encircles a part of the
critical graph $\Gamma_{f}.$ Looking at each of these curves as a $\varpi
$-polygon and applying the so-called Teichm\"{u}ller Lemma (see
\cite[Thm.14.1]{Strebel}), we get for $k=1,...,s$ :
\[
0=2+\sum n_{k}.
\]
Making the sum of all equalities, we obtain%
\[
0=2s+2\left(  n-1-n\right)  =2s-2;
\]
a contradiction.
\end{proof}

\section{Fingerprints of sub-lemniscates}

Let $\gamma$ be a $\mathcal{C}^{\infty}$ Jordan curve in $%
\mathbb{C}
;$ by a Jordan theorem, $\gamma$ splits $\widehat{%
\mathbb{C}
}$ into a bounded and an unbounded simply connected components $D_{-}$ and
$D_{+}.$ The Riemann mapping theorem asserts that there exist conformal maps
$\phi_{-}:\Delta\longrightarrow$ $D_{-},$ and $\phi_{+}:\widehat{%
\mathbb{C}
}\setminus\overline{\Delta}\longrightarrow$ $D_{+},$ where $\Delta$ is the
unit disk. The map $\phi_{+}$ is uniquely determined by the normalization
$\phi_{+}\left(  \infty\right)  =\infty$ and $\phi_{+}^{\prime}\left(
\infty\right)  >0.$ It is well-known that $\phi_{-}$ and $\phi_{+}$ extend to
$\mathcal{C}^{\infty}$-diffeomorphisms on the closure of their respective
domain. The \textit{fingerprint of }$\gamma$ is the the map $k$ $:=$ $\phi
_{+}^{-1}\circ\phi_{-}:S^{1}\longrightarrow S^{1}$ from the unit circle
$S^{1}$ to itself. Note that $k$ is uniquely determined by up to
post-composition with an automorphism of $D$ onto itself. Moreover, the
fingerprint $k$ is invariant under translations and scalings of the curve
$\gamma.$ Fingerprints of polynomial and rational "proper" lemniscates have
been studied in \cite{khavinson},\cite{frolova}, \cite{younsi}, and
\cite{Solynin}. Recall that a lemniscate is called "proper" if it is smooth
and connected. In the this paper, we investigate the case of smooth and not
necessary connected lemniscates. 

\bigskip Let $\lambda>0,\lambda\notin\left\{  w_{k},k=1,...,n-1\right\}  .$
Since the quadratic differential $\varpi$ is Strebel, the lemniscate
$\Gamma_{\lambda}$ is the union of a finite number of disjoint sub-lemniscates
in $%
\mathbb{C}
,$ each of them is entirely included either in a Circle or a Ring Domain of
$\varpi$. More precisely, if $\gamma$ is a sub-lemniscate of $\varpi,$ then
\begin{align*}
D_{-}\cap\left\{  a_{j}\right\}  _{1\leq j\leq n} &  \neq\emptyset,\\
D_{+}\cap\left(  \left\{  \infty\right\}  \cup\left\{  a_{j}\right\}  _{1\leq
j\leq n}\right)   &  \neq\emptyset.
\end{align*}
It's obvious (by definition of the Circle Domains) that, $\gamma$ is entirely
included in a Circle Domain of $\varpi,$ if and only if
\[
card\left(  D_{-}\cap\left\{  a_{j}\right\}  _{1\leq j\leq n}\right)
=1,\text{or }D_{+}\cap\left(  \left\{  \infty\right\}  \cup\left\{
a_{j}\right\}  _{1\leq j\leq n}\right)  =\left\{  \infty\right\}  .
\]
Observe that a lemniscate is proper if and only if it is contained in the
circle domain of $\varpi$ containing the infinity $\infty.$

\subsection{Lemniscates in a Circle Domain}

\begin{theorem}
Let be $a$ a pole $\varpi$ ( $a\in\left\{  \infty\right\}  \cup\left\{
a_{j}\right\}  _{1\leq j\leq n}$ ) and $\gamma_{a}$ a sub-lemniscate of
$\varpi.$ Then, the fingerprint $k$ $:S^{1}\longrightarrow S^{1}$ of
$\gamma_{a}$ is given by
\[
\left\{
\begin{array}
[c]{c}%
k(z)=B_{\infty}(z)^{1/\alpha},\text{ \emph{if }}a=\infty\\
k^{-1}(z)=z^{\alpha/\alpha_{j}}B_{j}(z)^{1/\alpha_{j}},\text{ \emph{if}
}a=a_{j}.
\end{array}
\right.
\]
with%
\[
B_{\infty}\left(  z\right)  =e^{i\theta_{\infty}}\prod_{i=1}^{n}\left(
\frac{z-\phi_{-}^{-1}\left(  a_{i}\right)  }{1-\overline{\phi_{-}^{-1}\left(
a_{i}\right)  }z}\right)  ^{\alpha_{i}},
\]
and%
\[
B_{j}\left(  z\right)  =e^{i\theta_{j}}\prod_{i\neq j}\left(  \frac{z-\phi
_{+}^{-1}\left(  a_{i}\right)  }{1-\overline{\phi_{+}^{-1}\left(
a_{i}\right)  }z}\right)  ^{\alpha_{i}}.
\]
for some real numbers $\theta_{\infty}$ and $\theta_{j}.$
\end{theorem}

\bigskip

\begin{proof}
Jenkins Theorem on the Configuration Domains of the quadratic differential
$\varpi$ asserts that there exists a connected neighborhood $\mathcal{U}_{a}$
of $a$ (a Circle Domain of $\varpi$) bounded by finite critical trajectories
of $\varpi,$ such that all trajectories of $\varpi$ (sub-lemniscates) inside
$\mathcal{U}_{a}$ are closed smooth curves encircling $a.$ Moreover, for a
suitably chosen non-vanishing real constant $c,$ the function
\[
\psi_{a}:z\longmapsto\exp\left(  c\sum_{i=1}^{n}\int^{z}\frac{\alpha_{i}%
}{z-a_{i}}dt\right)
\]
is a conformal map from $\mathcal{U}_{a}$ onto a certain disk centered in
$z=0.$ \begin{figure}[th]
\centering\includegraphics[height=1.8in,width=2.8in]{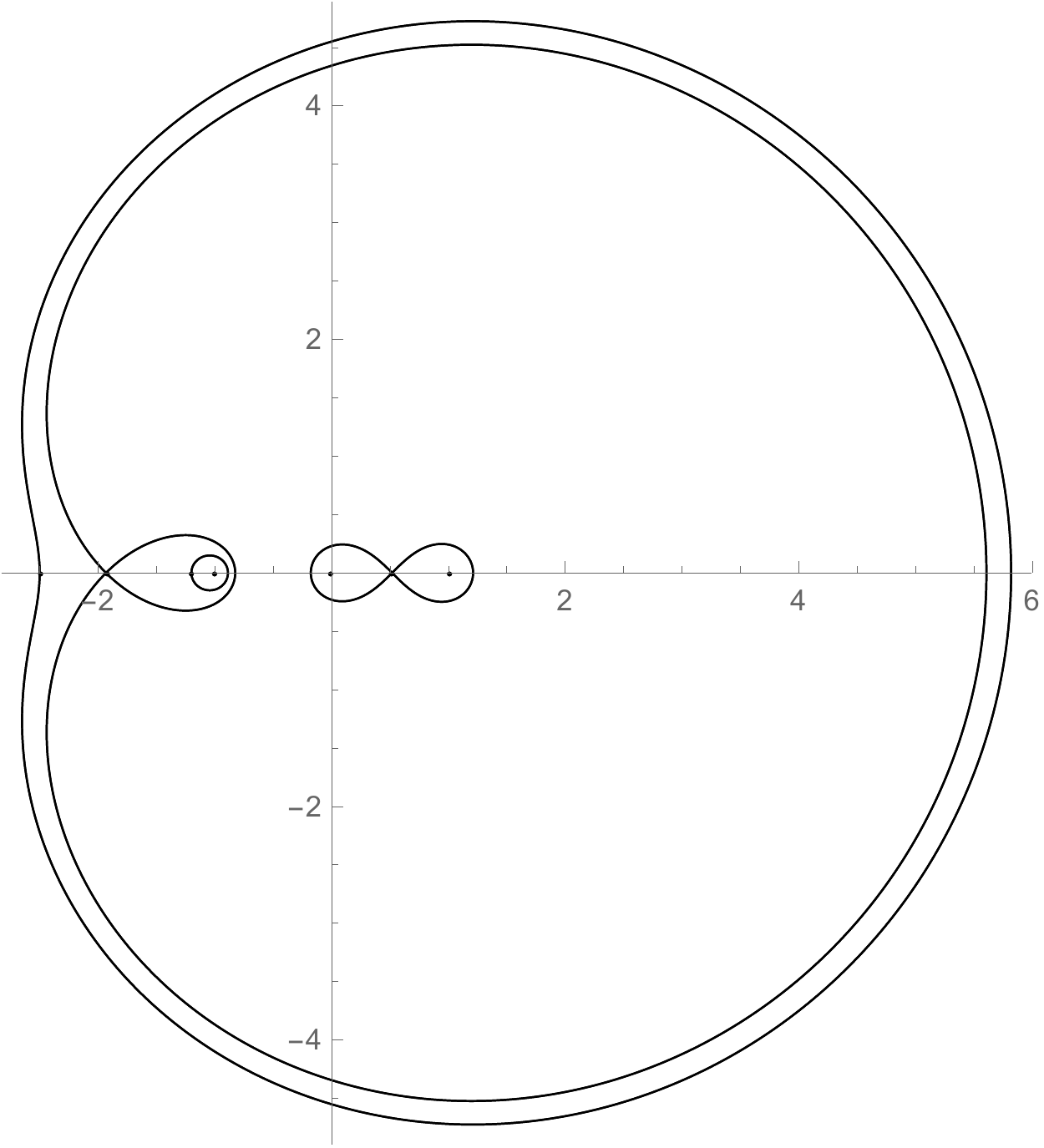}\caption{Critical
graph of $-\left(  \frac{1}{z-1}-\frac{1}{z+1}+\frac{\sqrt{2}}{z}\right)
^{2}dz^{2}.$ and lemniscates in Circle Domains: $a=\infty$, and $a=-1$.}%
\label{FIG2}%
\end{figure}A more obvious form of it, is
\[
\psi_{a}\left(  z\right)  =\beta\prod_{i=1}^{n}\left(  z-a_{i}\right)
^{c\alpha_{i}}%
\]
for some complex number $\beta.$ Baring in mind that $\psi$ is univalent near
$a$, we get
\[
c=\left\{
\begin{array}
[c]{c}%
\frac{1}{\alpha},\text{ if }a=\infty\\
\frac{1}{\alpha_{j}},\text{ if }a=a_{j}.
\end{array}
\right.
\]
It follows that the function
\[
z\longmapsto\left\{
\begin{array}
[c]{c}%
\prod_{i=1}^{n}\left(  z-a_{i}\right)  ^{\alpha_{i}/\alpha},\text{ if
}a=\infty,\\
\prod_{i=1}^{n}\left(  z-a_{i}\right)  ^{\alpha_{i}/\alpha_{j}},\text{ if
}a=a_{j}.
\end{array}
\right.
\]
is a conformal map from $\mathcal{U}_{a}$ onto a certain disk $\Delta_{a}$
centered in $z=0.$ For the sake of simplicity, we may assume that $\Delta_{a}$
with a radius $R>1.$ For the given sub-lemniscate $\gamma_{a}$ in
$\mathcal{U}_{a}$ (see Figure \ref{FIG2}), it is straightforward that the
previous function maps $\Omega_{-}$ conformally onto the unit disk $\Delta.$
Thus,%
\begin{equation}
\left\{
\begin{array}
[c]{c}%
\phi_{+}^{-1}\left(  z\right)  =f\left(  z\right)  ^{1/\alpha},\text{
\emph{if} }a=\infty,\\
\phi_{-}^{-1}\left(  z\right)  =f\left(  z\right)  ^{1/\alpha_{j}},\text{
\emph{if} }a=a_{j}.
\end{array}
\right.  . \label{equality}%
\end{equation}

\bigskip In the case $a=\infty,$ the functions%
\[
z\longmapsto\ \frac{\phi_{-}\left(  z\right)  -a_{i}}{\frac{z-\phi_{-}%
^{-1}\left(  a_{i}\right)  }{1-\overline{\phi_{-}^{-1}\left(  a_{i}\right)
}z}};\left\vert z\right\vert \leq1,i=1,...,n
\]
are holomorphic and non vanishing in the simply connected unit disk
$\overline{\Delta}.$It follows that the function
\[
z\longmapsto\ \frac{f\circ\phi_{-}\left(  z\right)  }{z^{\alpha}\prod
_{i=1}^{n}\left(  \frac{z-\phi_{-}^{-1}\left(  a_{i}\right)  }{1-\overline
{\phi_{-}^{-1}\left(  a_{i}\right)  }z}\right)  ^{\alpha_{i}}};\left\vert
z\right\vert \leq1
\]
is well defined, holomorphic and non vanishing in $\overline{\Delta},$ and has
modulus one in $S^{1}.$ Thus,
\[
\frac{f\circ\phi_{-}\left(  z\right)  }{\prod_{i=1}^{n}\left(  \frac
{z-\phi_{-}^{-1}\left(  a_{i}\right)  }{1-\overline{\phi_{-}^{-1}\left(
a_{i}\right)  }z}\right)  ^{\alpha_{i}}}=e^{i\theta};\left\vert z\right\vert
\leq1,
\]
for some real $\theta.$ From (\ref{equality}), we get
\[
\phi_{+}^{-1}\circ\phi_{-}\left(  z\right)  =B_{\infty}\left(  z\right)
^{1/\alpha};\left\vert z\right\vert =1.
\]

In the case $a=a_{j},$ with the normalization $\frac{\phi_{+}\left(  z\right)
}{z}\sim b>0,$ as $z\rightarrow\infty,$ the functions%
\[
z\longmapsto\left\{
\begin{array}
[c]{c}%
\frac{\phi_{+}\left(  z\right)  -a_{i}}{z\frac{z-\phi_{+}^{-1}\left(
a_{i}\right)  }{1-\overline{\phi_{+}^{-1}\left(  a_{i}\right)  }z}},\text{
\emph{if} }i\neq j\\
\frac{\phi_{+}\left(  z\right)  -a_{j}}{z},\text{ \emph{if} }i=j
\end{array}
\right.  ;\left\vert z\right\vert \geq1
\]
are holomorphic and non vanishing in the simply connected $\widehat{%
\mathbb{C}
}\setminus\overline{\Delta},$ it follows that the function
\[
z\longmapsto\ \frac{f\circ\phi_{+}\left(  z\right)  }{z^{\alpha}\prod_{i\neq
j}\left(  \frac{z-\phi_{+}^{-1}\left(  a_{i}\right)  }{1-\overline{\phi
_{+}^{-1}\left(  a_{i}\right)  }z}\right)  ^{\alpha_{i}}};\left\vert
z\right\vert \geq1
\]
is well defined and holomorphic in $%
\mathbb{C}
\setminus\overline{\Delta},$ does not vanish there, is continuous in $%
\mathbb{C}
\setminus\Delta,$ and has modulus one on $\partial\Delta=S^{1}$. We deduce the
existence of $\theta^{\prime}\in%
\mathbb{R}
$ such that
\[
f\circ\phi_{+}\left(  z\right)  =e^{i\theta^{\prime}}z^{\alpha}\prod_{i\neq
j}\left(  \frac{z-\phi_{+}^{-1}\left(  a_{i}\right)  }{1-\overline{\phi
_{+}^{-1}\left(  a_{i}\right)  }z}\right)  ^{\alpha_{i}};\left\vert
z\right\vert \geq1.
\]
Combining with (\ref{equality}), we get%
\[
\phi_{-}^{-1}\circ\phi_{+}\left(  z\right)  =z^{\alpha/\alpha_{j}}B_{j}\left(
z\right)  ^{1/\alpha_{j}};\left\vert z\right\vert =1.
\]

\end{proof}

\subsection{Lemniscates in a Ring Domain}

In the following, let $\mathcal{U}$ be a Ring Domain of the quadratic
differential $\varpi.$ It is bounded by two sub-lemniscates $\Gamma_{r}$ and
$\Gamma_{R}.$ We may assume that
\[
0<r<1<R.
\]
\begin{figure}[th]
\centering\includegraphics[height=1.8in,width=2.8in]{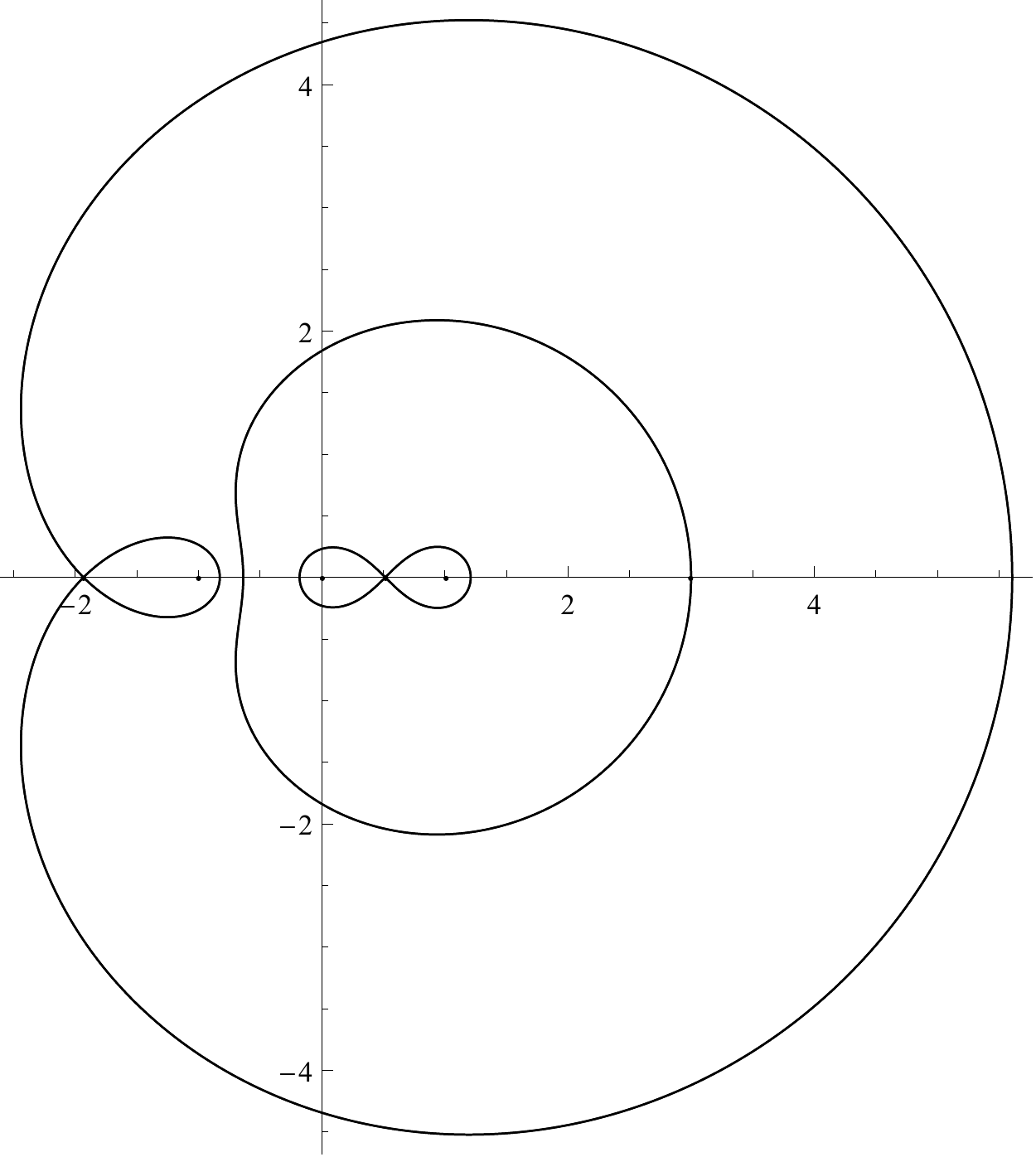}\caption{Critical
graph of $-\left(  \frac{1}{z-1}-\frac{1}{z+1}+\frac{\sqrt{2}}{z}\right)
^{2}dz^{2}$ with a lemniscate in a Ring Domain ( $a=1,b=0$ ).}%
\label{FIG3}%
\end{figure}Suppose that $a_{1},...,a_{s}$ $\left(  1<s<n\right)  $ are in the
bounded domain of $%
\mathbb{C}
$ with boundary $\Gamma_{r}.$ We consider the sub-lemniscate $\Gamma_{1}$ of
$\varpi$ in $\mathcal{U}$ (see Figure \ref{FIG3}).

Since the function
\[
z\longmapsto\frac{f\circ\phi_{-}\left(  z\right)  }{\prod_{i=1}^{s}\left(
\frac{z-\phi_{-}^{-1}\left(  a_{i}\right)  }{1-\overline{\phi_{-}^{-1}\left(
a_{i}\right)  }z}\right)  ^{\alpha_{i}}}%
\]
is holomorphic in $\Delta$, is continuous in $\overline{\Delta},$ does not
vanish in $\Delta,$ and has modulus one on $\partial\Delta,$ we deduce that
there exists $\theta_{1}\in%
\mathbb{R}
$ such that$\bigskip$%
\[
\frac{f\circ\phi_{-}\left(  z\right)  }{\prod_{i=1}^{s}\left(  \frac
{z-\phi_{-}^{-1}\left(  a_{i}\right)  }{1-\overline{\phi_{-}^{-1}\left(
a_{i}\right)  }z}\right)  ^{\alpha_{i}}}=\bigskip e^{i\theta_{1}};\left\vert
z\right\vert \leq1,
\]
and then
\[
f\circ\phi_{-}\left(  z\right)  =\bigskip e^{i\theta_{1}}\prod_{i=1}%
^{s}\left(  \frac{z-\phi_{-}^{-1}\left(  a_{i}\right)  }{1-\overline{\phi
_{-}^{-1}\left(  a_{i}\right)  }z}\right)  ^{\alpha_{i}};\left\vert
z\right\vert \leq1.
\]
Reasoning like in the previous subsection on $\phi_{+}\left(  z\right)  $ in
the case $a=a_{j},$ we get for some $\theta_{2}\in%
\mathbb{R}
$
\[
f\circ\phi_{+}\left(  z\right)  =\bigskip e^{i\theta_{2}}z^{\alpha}%
\prod_{i=s+1}^{n}\frac{z-\phi_{+}^{-1}\left(  a_{i}\right)  }{1-\overline
{\phi_{+}^{-1}\left(  a_{i}\right)  }z};\left\vert z\right\vert \geq1.
\]
Combining the last two equalities for $\left\vert z\right\vert =1,$ with
$\theta_{1}-\theta_{2}=\theta,$ we obtain the following

\begin{theorem}
Let $\Gamma_{1}$ be a sub-lemniscate of $\varpi$ such that $\Omega_{-}$
contains exactly two different zeros $a$ and $b$ of $p$ with respective
multiplicities $\alpha$ and $\beta.$The fingerprint $k$ $:S^{1}\longrightarrow
S^{1}$ of $\Gamma_{1}$ satisfies the functional equation
\[
\left(  B\circ k\right)  \left(  z\right)  =e^{i\theta}A\left(  z\right)
;\left\vert z\right\vert =1.
\]
where $A$ and $B$ are given by
\[
\bigskip B\left(  z\right)  =\bigskip z^{\alpha}\prod_{i=s+1}^{n}\frac
{z-\phi_{+}^{-1}\left(  a_{i}\right)  }{1-\overline{\phi_{+}^{-1}\left(
a_{i}\right)  }z},
\]%
\[
A\left(  z\right)  =\prod_{i=1}^{s}\left(  \frac{z-\phi_{-}^{-1}\left(
a_{i}\right)  }{1-\overline{\phi_{-}^{-1}\left(  a_{i}\right)  }z}\right)
^{\alpha_{i}}.
\]

\end{theorem}

\begin{acknowledgement}
Part of this work was partially supported by the Laboratory of Mathematics and
Applications "LR17ES11" from the Faculty of Sciences of Gab\`{e}s, Tunisia.
\end{acknowledgement}

\bigskip

\texttt{Higher Institute of Applied Sciences and Technology of Gabes, }

\texttt{Avenue Omar Ibn El Khattab, 6029. Tunisia.}

\texttt{E-mail adresses:}

\texttt{braekgliaa@gmail.com}

\texttt{faouzithabet@yahoo.fr}

\end{document}